\documentclass[a4paper,10pt]{article}
\usepackage{amsmath,amsfonts,amssymb,amscd}
\usepackage{indentfirst,graphicx,epsfig}
\usepackage{graphicx,psfrag}
\usepackage{amsthm, amsfonts, color}
\usepackage[bookmarksnumbered, plainpages]{hyperref}
\input{epsf}
\usepackage{fixmath}
\usepackage{tikz}
\usepackage{mathrsfs}
\usepackage{mathtools}
\usepackage{stmaryrd}
\usepackage{psfrag}
\usepackage{epsf}
\usepackage{url,rotating}
\usepackage[cp1250]{inputenc}
\usepackage{amsmath,amsfonts,amssymb,latexsym}
\usepackage{setspace}
\usepackage{enumitem}
\usepackage{multirow}
\title{The energy of random signed graph}
\author{\small{Shuchao Li\thanks{S. Li was partially supported by the National Natural Science Foundation of China (Grant Nos. 11671164, 11271149)
}, Shujing Wang\thanks{Corresponding author}}\\
{\small  School of Mathematics and Statistics}\\
{\small Central China Normal University, Wuhan 430079, P.R. China}\\
{\small Email: lscmath@mail.ccnu.edu.cn, wang06021@126.com}}
\date{}

\begin{document}

\makeatletter
  \newcommand\figcaption{\def\@captype{figure}\caption}
  \newcommand\tabcaption{\def\@captype{table}\caption}
\makeatother
\newtheorem{pro}{Proposition}[section]
\newtheorem{defi}{Definition}[section]
\newtheorem{theorem}{Theorem}[section]
\newtheorem{lemma}[theorem]{Lemma}
\newtheorem{coro}[theorem]{Corollary}

\maketitle

\begin{abstract}
A signed graph $\Gamma(G)$ is a graph with a sign attached to each of its
edges, where $G$ is the underlying graph of $\Gamma(G)$. The energy of a signed graph $\Gamma(G)$ is the sum of the absolute values of the eigenvalues of the adjacency matrix $A(\Gamma(G))$ of $\Gamma(G)$. The random signed graph model $\mathcal{G}_n(p, q)$ is defined as follows: Let $p, q \ge 0$ be fixed, $0 \le p+q \le 1$. Given a set of $n$ vertices, between each pair of distinct vertices there is either a positive edge with probability $p$ or a negative edge with probability $q$, or else there is no edge with probability $1-(p+ q)$. The edges between different pairs of vertices are chosen
independently. In this paper, we obtain an exact estimate of energy for almost all signed graphs. Furthermore, we establish lower and upper bounds to the energy of random multipartite signed graphs. \\[2mm]

\noindent{\bf Keywords:}energy, random signed graph, empirical spectra distribution, random multipartite signed graph.

\noindent{\bf AMS subject classification 2010:} 05C50, 15A48

\end{abstract}

\section{Introduction}
Let $G$ be a simple graph with vertex set $V(G)=\{v_1, v_2,\ldots, v_n\}$ and
edge set $E(G)$. Let $A(G)$ be the adjacency matrix of $G$ and $\lambda_1(G)\ge \lambda_2(G)\ge\ldots\ge \lambda_n(G)$ be the eigenvalues of $A(G)$. The \emph{energy} of $G$ is defined as
$$\mathcal{E}(G)=\sum_{i=1}^n|\lambda_{i}(G)|.$$
This graph invariant is derived from the total $\pi$-electron energy \cite{Y} from chemistry and was first introduced by Gutman \cite{Gutman} in 1978. Since then, graph energy has been studied extensively by lots of mathematicians and chemists. For results on the study of the energy of graphs, we refer the reader to the books \cite{GL, LSG}.

A signed graph $\Gamma(G)=(G,\sigma)$ is a graph with a sign attached to each of its edges, and consists of a simple graph $G=(V,E)$, referred to as its underlying graph, and a mapping $\sigma: E\rightarrow\{+,-\}$, the edge labeling. To avoid confusion, we also write $V(\Gamma(G))$ instead of $V$, $E(\Gamma(G))$ instead of $E$, and $E(\Gamma(G))=E^{\sigma}$. Signed graphs were introduced by Harary \cite{Harary} in connection with the study of the theory
of social balance in social psychology (see \cite{De}). The matroids of graphs were extended to those of signed graphs by Zaslavsky \cite{Za}, and the Matrix-Tree Theorem for signed graphs
was obtained by Zaslavsky \cite{Za}, Chaiken \cite{Ch}, and also by Belardo and Simi\'c \cite{Be}.

The \textit{adjacency matrix} of $\Gamma(G)$ is $A(\Gamma(G))=(a_{ij}^{\sigma})$ with $a_{ij}^{\sigma}=\sigma (v_iv_j)a_{ij}$, where $(a_{ij})_{n\times n}$ is the adjacency matrix of the underlying graph $G$. In the case of $\sigma=+$, which is an all-positive edge labeling, $A(G,+)$ is exactly the classical adjacency matrix of $G$. So a simple graph is always assumed as a signed graph with all edges positive. The concept of energy of a graph was extended to signed graphs by Germina, Hameed
and Zaslavsky \cite{za} and they defined the \emph{energy} of a signed graph $\Gamma(G)$ to be the sum of absolute values of eigenvalues of the eigenvalues of $A(\Gamma(G))$, i.e.,
$$\mathcal{E}(\Gamma(G))=\sum_{i=1}^n|\lambda_i(\Gamma(G))|,$$
where $\lambda_1(\Gamma(G))\ge \lambda_2(\Gamma(G))\ge \ldots\ge \lambda_n(\Gamma(G))$ are the eigenvalues of $A(\Gamma(G))$.

Evidently, one can immediately get the energy of a graph (resp. signed graph) by computing the eigenvalues of the matrices $A(G)$ (resp. $A(\Gamma(G))$). It is rather hard, however, to compute the eigenvalues for a large matrix, even for a large symmetric matrix. So many researchers established a lot of lower and upper bounds to estimate the invariant for some classes of graphs among which the bipartite
graphs are of particular interest. For further details, we refer the readers to the comprehensive survey \cite{GLZ}.

In 1950s, Erd\v os and R\'enyi \cite{ER} founded the theory of random graphs. The Erd\v os-R\'enyi random graph
model $\mathcal{G}_n(p)$ consists of all graphs on $n$ vertices in which the edges are chosen independently with probability $p$, where $p$ is a constant and $0 < p < 1$. In \cite{Li1} and \cite {Li3}, Du, Li and Li have considered the energy of the Erd\v os and R\'enyi model $G_n(p)$ and have the following result.
\begin{theorem}\label{Th1}
Almost every random graph $G_n(p)$ enjoys the equation as follows:
$$E(G_n(p))=n^{3/2}\left(\frac{8}{3\pi}\sqrt{p-p^2}+o(1)\right).$$
\end{theorem}

In \cite{balance}, the authors defined a probabilistic model in which relations between individuals are assumed to be random. A good mathematical model for representing such random social structures is the so-called \emph{random signed graph} $G_n(p, q)$ defined as follows. Let $p, q \ge0$ be fixed, $0 \le p+q \le 1$. Given a set of $n$ vertices, between each pair of distinct vertices there is either a positive edge with probability $p$ or a negative edge with probability $q$, or else there is no edge with probability $1-(p + q)$. The edges between different pairs of vertices are chosen
independently.
Throughout this article the expression ``almost surely'' (a.s.) means ``with probability tending to 1 as $n$ tends to infinity.''

In this paper, we shall obtain an exact estimate of energy for almost all signed graphs. Furthermore, we establish lower and upper bounds to the energy of random multipartite signed graphs.
\section{The energy of random signed graph}
Let $X$ be an $n\times n$ Hermitian matrix
and denote its eigenvalues by $\lambda_i(X), i=1,2,\ldots,n$. \textit{The empirical spectral distribution (ESD)} of $X$ is defined by
$$
\Phi_{X}(x) =\frac{1}{n}\#\{\lambda_{i}(X)| \lambda_{i}(X)\le x, i=1,2,\ldots,n\},
$$
where $\#\{.\}$ denotes the number of elements in the set indicated.

By the definition of ESD of a matrix $X$, the following two facts are obvious:
\begin{itemize}
  \item Fact 1:For any positive $k$, $\int x^kd\Phi_{X}(x)=\frac{1}{n}\sum_{i=1}^n\lambda^k_i(X)=\frac{1}{n}trace(X^k).$
  \item Fact 2:$\int |x|d\Phi_{X}(x)=\frac{1}{n}\sum_{i=1}^n|\lambda_i(X)|.$
\end{itemize}
By Fact 2, one can readily evaluate the energy of $G_n(p,q)$ once the ESD of $A(G_n(p,q))$ is known.

One of the main problems in the theory of random matrices (RMT) is to investigate the convergence of the sequence of empirical spectral distributions $\Phi_{X_n}$ for a given sequence of random matrices $\{X_n\}$. The limit distribution (possibly defective; that is, total mass is less than 1 when some eigenvalues tend to $\pm\infty$), which is
usually nonrandom, is called the \textit{limiting spectral distribution (LSD)} of the
sequence $\{X_n\}$.
In fact, the study on the spectral distributions of random matrices is rather abundant and active,
which can be traced back to \cite{W}. We refer the readers to \cite{Bai, D, M} for an overview and some spectacular progress in this field. One important achievement in that field is Wigner's semi-circle law which characterizes the limiting spectral distribution of the empirical spectral distribution of eigenvalues for a sort of random matrix.

In order to characterize the statistical properties of the wave functions of quantum mechanical systems, Wigner in 1950s investigated the spectral distribution for a sort of random matrix, so-called \textit{Wigner matrix},
$X_n:=(x_{ij}), 1 \le i, j \le n$,
which satisfies the following properties:
\begin{itemize}
  \item $x_{ij}$'s are independent random variables with $x_{ij}$ = $x_{ji}$;
  \item the $x_{ii}$'s have the same distribution $F_1$, while the $x_{ij}$'s ($i \neq j$) have the same distribution $F_2$;
  \item $Var(x_{ij})=\sigma_2^2<\infty$ for all $1\le i< j\le n$.
\end{itemize}

Wigner \cite{Wigner1, Wigner2} considered the limiting spectral distribution (LSD) of $X_n$, and obtained his semi-circle law.

\begin{theorem}\label{Wigner}
Let $X_n$ be a Wigner matrix. Then
$$
\lim_{n\rightarrow\infty}\Phi_{n^{-1/2}X_n}(x)=\Phi_{0,\sigma_2}(x) \ \  a.s.
$$
i.e., with probability 1, the ESD $\Phi_{n^{-1/2}X_n}(x)$ converges weakly to a distribution $\Phi_{0,\sigma_2}(x)$ as $n$ tends to infinity, where $\Phi_{0,\sigma_2}(x)$ has the density
$$
\phi_{0,\sigma_2}(x)=\left\{
          \begin{array}{ll}
            \frac{1}{2\pi\sigma_2^2}\sqrt{4\sigma_2^2-x^2}, & \hbox{if $|x|\leq 2\sigma_2$;} \\
            0, & \hbox{otherwise.}
          \end{array}
        \right.
$$
\end{theorem}
Employing the moment approach, one can show that
\begin{lemma}\cite{Li1}\label{lemma1}
Let $X_n$ be a Wigner matrix such that each entry of $X_n$ has mean 0, then for each positive integer $k$,
\begin{equation}\label{1}
    \lim_{n\rightarrow\infty}\int x^kd\Phi_{n^{-1/2}X_n}(x)=\int x^kd\Phi(x) \ \ \ a.s.
\end{equation}
\end{lemma}
Let $G_n(p,q)$ be a random signed graph in $\mathcal{G}(n,p,q)$ and $A(G_n(p,q))$ be the adjacency matrix of $G_n(p,q)$. Evidently, $A(G_n(p,q))$ is a symmetric matrix and the entries of $A(G_n(p,q))=(s_{ij})_{n\times n}$ satisfy that for $i=1,2,\ldots,n$, $s_{ii}=0$ and for $1\le i\neq j\le n$, $s_{ij}$'s are independent random variables and have the same distribution with mean $p-q$ and $Var(s_{ij})=\sigma^2=p+q-(p-q)^2<\infty$.

Set $\overline{S}=A(G_n(p,q))-(p-q)(A(K_n))$, where $A(K_n)=J_n-I_n$ be the adjacency matrix of the complete graph $K_n$. One can see that each entry of $\overline{S}$ has mean 0 and variance $\sigma^2=p+q-(p-q)^2<\infty$ for all $1\le i< j\le n$. Then by Theorem \ref{Wigner}, we have that
\begin{equation}\label{A}
    \lim_{n\rightarrow\infty}\Phi_{n^{-1/2}\overline{S}}(x)=\Phi_{0,\sigma}(x) \ \  a.s.
\end{equation}
Furthermore, as each entry of $\overline{S}$ has mean 0, by Lemma \ref{lemma1}, we have that for any positive integer $k$,
\begin{equation}\label{B}
    \lim_{n\rightarrow\infty}\int x^kd\Phi_{n^{-1/2}\overline{S}}(x)=\int x^kd\Phi_{0,\sigma}(x) \ \ \ a.s.
\end{equation}
where $\Phi_{0,\sigma}(x)$ has the density $\phi_{0,\sigma}(x)$.
%

Let $I$ be the interval $[-2,2]$.
\begin{lemma}\label{lemma2}
Let $I^c$ be the set $\mathbb{R}\setminus I$. Then
\begin{equation}\label{C}
    \lim_{n\rightarrow\infty}\int_{I^c} x^2d\Phi_{n^{-1/2}\overline{S}}(x)=\int_{T^c} x^2d\Phi_{0,\sigma}(x) \ \ \ a.s.
\end{equation}
\end{lemma}
\begin{proof}
Suppose $\phi_{n^{-1/2}\overline{S}}$ is the density of $\Phi_{n^{-1/2}\overline{S}}$. According to Eq. \ref{A}, with probability 1,
$\phi_{n^{-1/2}\overline{S}}$ converges to $\phi_{0,\sigma}(x)$ almost everywhere as $n$ tends to infinity. Since $\phi_{0,\sigma}(x)$ is bounded on $[-2\sigma, 2\sigma]\subset I$, it follows that with probability 1, $x^2\phi_{n^{-1/2}\overline{S}}$ bounded almost everywhere on $I$. Invoking bounded convergence theorem yields
$$
\lim_{n\rightarrow\infty}\int_{I} x^2d\Phi_{n^{-1/2}\overline{S}}(x)=\int_{I} x^2d\Phi_{0,\sigma}(x) \ \ \ a.s.
$$
Combining the above fact with Eq. \ref{B}, we have
\begin{eqnarray*}
  \lim_{n\rightarrow\infty}\int_{I^C} x^2d\Phi_{n^{-1/2}\overline{S}}(x) &=&
  \lim_{n\rightarrow\infty}\left(\int x^2d\Phi_{n^{-1/2}\overline{S}}(x)-\int_{I} x^2d\Phi_{n^{-1/2}\overline{S}}(x)\right) \\
   &=& \lim_{n\rightarrow\infty}\int x^2d\Phi_{n^{-1/2}\overline{S}}(x)-\lim_{n\rightarrow\infty}\int_{I} x^2d\Phi_{n^{-1/2}\overline{S}}(x) \\
   &=& \int x^2d\Phi_{0,\sigma}(x)-\int_{I} x^2d\Phi_{0,\sigma}(x)\\
   &=& \int_{I^C} x^2d\Phi_{0,\sigma}(x)\ \ a.s.
\end{eqnarray*}
\end{proof}
To investigate the convergence of $\int |x|d\Phi_{n^{-1/2}\overline{S}}(x)$, we need the following result:
\begin{lemma}\label{lemma3}\cite{Bi}
 Let $\mu$ be a measure. Suppose that functions $a_n, b_n, f_n$ converges almost
everywhere to functions $a, b, f$, respectively, and that $a_n \le f_n \le b_n$ almost everywhere. If $ \int a_n d\mu\rightarrow \int a d\mu$
and $\int b_n d\mu \rightarrow \int b d\mu$, then $\int f_n d\mu\rightarrow \int f d\mu$.
\end{lemma}
By using Lemmas \ref{lemma2} and \ref{lemma3}, we have that
\begin{lemma}\label{lemma4}
Let $\overline{S}$ be the matrix defined above, then
\begin{equation}\label{D}
    \lim_{n\rightarrow\infty}\int |x|d\Phi_{n^{-1/2}\overline{S}}(x)=\int |x|d\Phi_{0,\sigma}(x) \ \ \ a.s.
\end{equation}
\end{lemma}
\begin{proof}
Suppose $\phi_{n^{-1/2}\overline{S}}$ is the density of $\Phi_{n^{-1/2}\overline{S}}$. According to Eq. \ref{A}, with probability 1,
$\phi_{n^{-1/2}\overline{S}}$ converges to $\phi_{0,\sigma}(x)$ almost everywhere as $n$ tends to infinity. Since $\phi_{0,\sigma}(x)$ is bounded on $[-2\sigma, 2\sigma]\subset I$, it follows that with probability 1, $|x|\phi_{n^{-1/2}\overline{S}}$ bounded almost everywhere on $I$. Invoking bounded convergence theorem yields
$$
\lim_{n\rightarrow\infty}\int_{I} |x|d\Phi_{n^{-1/2}\overline{S}}(x)=\int_{I} |x|d\Phi_{0,\sigma}(x) a.s.
$$
Obviously, $0\le |x|\le x^2$ if $x\in I^c$. Set $a_n=0, b_n=x^2\phi_{n^{-1/2}\overline{S}}$ and $f_n=|x|\phi_{n^{-1/2}\overline{S}}$. By Lemma \ref{lemma2} and Lemma \ref{lemma3}, we have
$$
\lim_{n\rightarrow\infty}\int_{I^c} |x|d\Phi_{n^{-1/2}\overline{S}}(x)=\int_{I^c} |x|d\Phi_{0,\sigma}(x) a.s.
$$
Consequently,
\begin{eqnarray*}
  \lim_{n\rightarrow\infty}\int |x|d\Phi_{n^{-1/2}\overline{S}}(x) &=&
  \lim_{n\rightarrow\infty}\int_{I} |x|d\Phi_{n^{-1/2}\overline{S}}(x)+\lim_{n\rightarrow\infty}\int_{I^c} |x|d\Phi_{n^{-1/2}\overline{S}}(x) \\
   &=&  \int_{I} |x|d\Phi_{0,\sigma}(x)+\int_{I^c} |x|d\Phi_{0,\sigma}(x)\\
   &=& \int |x|d\Phi_{0,\sigma}(x) \ \ \ a.s.
\end{eqnarray*}
\end{proof}
We define \emph{the energy $\mathcal{E}(M)$ of a matrix} $M$ as the sum
of absolute values of the eigenvalues of $M$. By fact 2 and Lemma \ref{lemma4}, we can deduce that
\begin{eqnarray*}
  \frac{1}{n}\mathcal{E}(n^{-1/2}\overline{S}) &=& \int |x|d\Phi_{n^{-1/2}\overline{S}}(x) \\
   &\rightarrow& \int |x|d\Phi_{0,\sigma}(x) \ \ \  a.s. \ \ \ \ (n\rightarrow\infty) \\
   &=& \frac{1}{2\pi\sigma^2}\int_{-2\sigma}^{2\sigma}|x|\sqrt{4\sigma^2-x^2}dx \\
   &=& \frac{8}{3\pi}\sigma=\frac{8}{3\pi} \sqrt{p+q-(p-q)^2}.
\end{eqnarray*}

Therefore, a.s. the energy of $\overline{S}$ enjoys the equation as follows:
\begin{equation}\label{F}
    \mathcal{E}(\overline{S})=n^{1/2}\mathcal{E}(n^{-1/2}\overline{S})
    =n^{3/2}\left(\frac{8}{3\pi} \sqrt{p+q-(p-q)^2}+o(1)\right).
\end{equation}
\begin{lemma}\cite{Fan}\label{lemma5}
Let $X, Y, Z$ be real symmetric matrices of order $n$ such that $X+Y=Z$. Then
$$
\mathcal{E}(X)+\mathcal{E}(Y)\ge \mathcal{E}(Z).
$$
\end{lemma}
It is easy to see that the eigenvalues of $A(K_n)$ are $n-1$ and -1 of $n-1$ times. Consequently, $\mathcal{E}((p-q)A(K_n))=2|p-q|(n-1).$
Note that $A(G_n(p,q))=\overline{S}+(p-q)A(K_n)$, it follows from Lemma \ref{lemma5} that with probability 1,
\begin{eqnarray*}
  \mathcal{E}(A(G_n(p,q))) &\le& \mathcal{E}(\overline{S})+\mathcal{E}((p-q)A(K_n)) \\
  &=& n^{3/2}\left(\frac{8}{3\pi} \sqrt{p+q-(p-q)^2}+o(1)\right)+2|p-q|(n-1).
\end{eqnarray*}
Consequently,
\begin{equation}\label{G}
\lim_{n\rightarrow\infty}\mathcal{E}(A(G_n(p,q)))/n^{3/2}\le
\left(\frac{8}{3\pi} \sqrt{p+q-(p-q)^2}+o(1)\right).
\end{equation}

One the other hand, as $\overline{S}=A(G_n(p,q))-(p-q)A(K_n)$, by Lemma \ref{lemma5}, that with probability 1,
\begin{eqnarray*}
  \mathcal{E}(A(G_n(p,q))) &\ge& \mathcal{E}(\overline{S})-\mathcal{E}((p-q)A(K_n)) \\
  &=& n^{3/2}\left(\frac{8}{3\pi} \sqrt{p+q-(p-q)^2}+o(1)\right)-2|p-q|(n-1).
\end{eqnarray*}
Thus we have that with probability 1,
\begin{equation}\label{H}
\lim_{n\rightarrow\infty}\mathcal{E}(A(G_n(p,q)))/n^{3/2}\ge
\left(\frac{8}{3\pi} \sqrt{p+q-(p-q)^2}+o(1)\right).
\end{equation}
Combining Eq. \ref{G} with Eq. \ref{H}, we obtain the energy of random signed graph $G_n(p,q)$ immediately. 
\begin{theorem}\label{Th2}
Almost every random signed graph $G_n(p,q)$ enjoys the equation as follows:
$$
\mathcal{E}(A(G_n(p,q)))=
n^{3/2}\left(\frac{8}{3\pi} \sqrt{p+q-(p-q)^2}+o(1)\right). $$
\end{theorem}
\textbf{Remark 1.}If $q=0$, it is obvious that $\mathcal{G}_n(p,0)=\mathcal{G}_n(p)$. Thus by Theorem \ref{Th2}, we can get the energy of random graph $G_n(p)$, which can also be seen in \cite{Li1} and \cite{Li3}.
\begin{coro}
Almost every random graph $G_n(p)$ enjoys the equation as follows:
$$
\mathcal{E}(A(G_n(p)))=
n^{3/2}\left(\frac{8}{3\pi} \sqrt{p-p^2}+o(1)\right). $$
\end{coro}
Recall that a signed graph $G$ is balanced if each cycle of $G$ has even number of negative edges, otherwise it is unbalanced. The spectral criterion for the balance of signed graphs given by Acharya \cite{Acharya} is as follows:
\begin{lemma}\label{add}
A signed graph is balanced if and only if it is co-spectral with its underlying graph.
\end{lemma}
By Lemma \ref{add}, we can see that if a singed graph $\Gamma(G)$ is balanced only if $\mathcal{E}(\Gamma(G))=\mathcal{E}(G)$, where $G$ is the underlying graph of $\Gamma(G)$.
For signed random graph $G_n(p,q)\in \mathcal{G}_n(p,q)$, the underlying graph belongs to $\mathcal{G}_n(p+q)$. It is obvious that with probability 1,
\begin{eqnarray*}
  \frac{\mathcal{E}(A(G_n(p,q)))-\mathcal{E}(A(G_n(p+q)))}{n^{3/2}} &=& \left(\frac{8}{3\pi}\sqrt{p+q-(p-q)^2}-
  \frac{8}{3\pi}\sqrt{p+q-(p+q)^2}+o(1)\right) \\
   &=& \left(\frac{32pq}{3\pi(\sqrt{p+q-(p-q)^2}+\sqrt{p+q-(p+q)^2})}+o(1)\right)
\end{eqnarray*}

Thus we have that if $p\neq0$ and $q\neq0$, the energy of random signed graph $G_n(p,q)$ is larger than the energy of its underlying graph. Hence by Lemma \ref{add}, we can get the following result which has been proved in \cite{balance}.
\begin{theorem}
Almost every random signed graph $G_n(p,q)$ is unbalanced.
\end{theorem}
\section{The energy of the random multipartite signed graph}
We begin with the definition of the random multipartite graph and random multipartite signed graph. We use $K_{n;\beta_1,\ldots,\beta_k}$ to denote the
complete $k$-partite graph with vertex set $V=\{v_1,v_2,\ldots,v_n\}$ whose parts $V_1, \ldots,V_k (k\ge2)$ are such that
$|V_i| = n\beta_i = n\beta_i(n), i = 1,\ldots, k$. The random $k$-partite graph model $\mathcal{G}_{n;\beta_1,\ldots,\beta_k}(p)$ consists of all random $k$-partite graphs in which the
edges are chosen independently with probability $p$ from the set of edges of $K_{n;\beta_1,\ldots,\beta_k}$. The random $k$-partite signed graph model $\mathcal{G}_{n;\beta_1,\ldots,\beta_k}(p,q)$ consists of all random $k$-partite signed graphs in which the
edges are chosen to be positive edges with probability $p$ or negative edges with probability $q$ from the set of edges of $K_{n;\beta_1,\ldots,\beta_k}$. We denote by
$A_{n,k}(p):= A(G_{n;\beta_1,\ldots,\beta_k}(p))$ the adjacency matrix of the random $k$-partite graph $G_{n;\beta_1,\ldots,\beta_k}(p)\in \mathcal{G}_{n;\beta_1,\ldots,\beta_k}(p)$ and
$S_{n,k}(p,q):=A(G_{n;\beta_1,\ldots,\beta_k}(p,q))= (s_{ij})_{n\times n}$ the adjacency matrix of the random $k$-partite signed graph $G_{n;\beta_1,\ldots,\beta_k}(p,q)\in \mathcal{G}_{n;\beta_1,\ldots,\beta_k}(p,q)$. We can see that $S_{n,k}(p,q)$ satisfies the following properties:
\begin{itemize}
  \item $s_{ij}$'s, $1 \le i < j \le n$, are independent random variables with $s_{ij} = s_{ji}$;
  \item for $i\in V_t$ and $j\in V\setminus V_t$, $\mathbb{P}(s_{ij}=1)=p, \mathbb{P}(s_{ij}=-1)=q$ and $\mathbb{P}(s_{ij}=0)=1-p-q$;
  \item for $i,j\in V_t$, $\mathbb{P}(s_{ij}=0)=1$.
\end{itemize}
The rest of this section will be divided into two parts. In the first part, we will establish lower and upper bounds of the
energy of the random multipartite signed graph $G_{n;\beta_1,\ldots,\beta_k}(p,q)$; In the second part, we will obtain an exact estimate of
the energy of the random bipartite signed graph $G_{n;\beta_1,\beta_2}(p,q)$.
\subsection{bounds of the energy of random multipartite signed graph}
\begin{theorem}
Let $G_{n;\beta_1,\ldots,\beta_k}(p,q)\in \mathcal{G}_{n;\beta_1,\ldots,\beta_k}(p,q)$ with $\lim_{n\rightarrow\infty}\beta_i(n)>0, i=1,2,\ldots,k$. Then almost surely
$$
n^{3/2}(1-\sum_{i=1}^k\beta_i^{3/2})
\le\mathcal{E}(G_{n;\beta_1,\ldots,\beta_k}(p,q))\left(\frac{8}{3\pi} \sqrt{p+q-(p-q)^2}+o(1)\right)^{-1}\le n^{3/2}(1+\sum_{i=1}^k\beta_i^{3/2}).
$$
\end{theorem}
\begin{proof}
Let $V_1, V_2, \ldots, V_k$ be the parts of the random multipartite signed graph $G_{n;\beta_1,\ldots,\beta_k}(p,q)$ satisfying $|V_i|=n\beta_i, i=1,2,\ldots,k$. Let $S_n(p,q)$ be the adjacency matrix of the random signed graph $G_n(p,q)$ and $S_{n,k}(p,q)$ be the adjacency matrix of $G_{n;\beta_1,\ldots,\beta_k}(p,q)$. It is obvious that
\begin{equation}\label{2.1}
   S_{n,k}(p,q)+S'_{n,k}(p,q)=S_n(p,q),
\end{equation}
where
$$
S'_{n,k}(p,q)=\begin{pmatrix}
                S_{n\beta_1}(p,q) &  &  &  \\
                 & S_{n\beta_2}(p,q) &  &  \\
                 &  & \ddots &  \\
                 &  &  & S_{n\beta_k}(p,q) \\
              \end{pmatrix}.
$$
Note that for $i=1,2,\ldots, k, \lim_{n\rightarrow\infty}\beta_i(n)>0$, thus by Theorem \ref{Th2}, we have that
$$
\mathcal{E}(S_{n\beta_i}(p,q))=\mathcal{E}(A(G_{n\beta_i}(p,q)))=
(n\beta_i)^{3/2}\left(\frac{8}{3\pi} \sqrt{p+q-(p-q)^2}+o(1)\right) \ \ a.s.
$$
Therefore,
\begin{equation}\label{2.2}
\mathcal{E}(S'_{n,k}(p,q))=\sum_{i=1}^k\mathcal{E}(S_{n\beta_i}(p,q))=
n^{3/2}\left(\frac{8}{3\pi} \sqrt{p+q-(p-q)^2}+o(1)\right)\sum_{i=1}^k\beta_i^{3/2} \ \ a.s.
\end{equation}

By Eq. \ref{2.1} and Lemma \ref{lemma5}, we have
\begin{equation}\label{2.3}
  |\mathcal{E}(S_n(p,q))-\mathcal{E}(S'_{n,k}(p,q))|\le\mathcal{E}(S_{n,k}(p,q))\le \mathcal{E}(S_n(p,q))+\mathcal{E}(S'_{n,k}(p,q))
\end{equation}
Recall that $0<\beta_i<1$ and $\sum_{i=1}^k\beta_i=1$, we can see that $\sum_{i=1}^k\beta_i^{3/2}<\sum_{i=1}^k\beta_i=1$.

Hence by Eq. \ref{2.2} and Eq. \ref{2.3}, we can deduce that almost every random multipartite signed graph $G_{n;\beta_1,\ldots,\beta_k}(p,q)$ satisfies the inequality
$$
n^{3/2}(1-\sum_{i=1}^k\beta_i^{3/2})
\le\mathcal{E}(G_{n;\beta_1,\ldots,\beta_k}(p,q))\left(\frac{8}{3\pi} \sqrt{p+q-(p-q)^2}+o(1)\right)^{-1}\le n^{3/2}(1+\sum_{i=1}^k\beta_i^{3/2}).
$$
\end{proof}
\subsection{the energy of random bipartite signed graph}
In what follows, we investigate the energy of random bipartite signed graphs $G_{n;\beta_1,\beta_2}(p,q)$ satisfying
$\lim_{n\rightarrow\infty}\beta_i(n) > 0 (i = 1, 2)$, and present the precise estimate of $\mathcal{E}(G_{n;\beta_1,\beta_2}(p,q))$ via Mar\v cenko-Pastur
Law.
For convenience, set $n_1= n\beta_1, n_2 =n\beta_2$ and $\beta_1/\beta_2\rightarrow y\in (0, \infty)$. Let $A(K_{n_1, n_2})$ be adjacency matrix of complete bipartite graph $K_{n_1, n_2}$ and set
\begin{equation}\label{eq3.1}
\overline{S}_{n,2}(p,q)=S_{n,2}(p,q)-(p-q)A(K_{n_1,n_2})=\begin{pmatrix}
                                                        0 & X \\
                                                        X^T & 0 \\
                                                      \end{pmatrix}
,
\end{equation}
where $X$ is a random matrix of order $n_1\times n_2$ in which the entries $X_{ij}$ are iid. (independent and identically distributed) with mean zero and variance $\sigma^2=p+q-(p-q)^2$.
In \cite{Bai}, Bai formulated the LSD of $\frac{1}{n_1}XX^T$ by moment approach.
\begin{lemma}\label{lemma3.2}\cite{Bai}
Suppose that $X$ is a random matrix of order $n_1\times n_2$ in which the entries $X_{ij}$ are iid. with mean zero and variance $\sigma^2$ and $\beta_1/\beta_2\rightarrow y\in (0, \infty)$. Then, with probability 1, the ESD $\Phi_{\frac{1}{n_1}XX^T}(x)$ converges weakly to the Mar\v cenko-Pastur Law $F_y(x)$
 as $n\rightarrow\infty$, where $F_y$ has the density
$$
f_y(x)=\left\{
          \begin{array}{ll}
            \frac{1}{2\pi xy\sigma^2}\sqrt{(b-x)(x-a)}, & \hbox{if $a\le x\le b$;} \\
            0, & \hbox{otherwise.}
          \end{array}
        \right.
$$
where $a=\sigma^2(1-\sqrt{y})^2$ and  $b=\sigma^2(1+\sqrt{y})^2$.
\end{lemma}
\begin{lemma}\label{lemma3.3}
Let $X$ be a real matrix of order $n_1\times n_2$ and $A=\begin{pmatrix}
                                                        0 & X \\
                                                        X^T & 0 \\
                                                      \end{pmatrix}.$
If $\lambda_1(XX^T)\ge \cdots \ge \lambda_{n_1}(XX^T)$ are the eigenvalues of $XX^T$, then
$$
\mathcal{E}(A)=2\sum_{i=1}^{n_1}\sqrt{\lambda_i(XX^T)}.
$$
\end{lemma}
\begin{proof}
Let $\begin{pmatrix}
                  Y_{1\times n_1} \\
                  Z_{1\times n_2} \\
                \end{pmatrix}$ be the eigenvector corresponding to the eigenvalue $\lambda$ of $A$. Then
$$
A\begin{pmatrix}
                  Y \\
                  Z \\
                \end{pmatrix}=\begin{pmatrix}
     0 & X \\
     X^T & 0 \\
   \end{pmatrix}\begin{pmatrix}
                  Y \\
                  Z \\
                \end{pmatrix}=\begin{pmatrix}
                  XZ \\
                  X^TY \\
                \end{pmatrix}=\lambda\begin{pmatrix}
                   Y \\
                  Z \\
                \end{pmatrix}=\begin{pmatrix}
                  \lambda Y \\
                 \lambda Z \\
                \end{pmatrix}
$$
i.e., $XZ=\lambda Y$ and $X^TY=\lambda Z$. Hence
$$
A\begin{pmatrix}
                  Y \\
                  -Z \\
                \end{pmatrix}=\begin{pmatrix}
     0 & X \\
     X^T & 0 \\
   \end{pmatrix}\begin{pmatrix}
                  Y \\
                  -Z \\
                \end{pmatrix}=\begin{pmatrix}
                  -XZ \\
                  X^TY \\
                \end{pmatrix}=\begin{pmatrix}
                  -\lambda Y \\
                 \lambda Z \\
                \end{pmatrix}=-\lambda\begin{pmatrix}
                   Y \\
                 -Z \\
                \end{pmatrix}
$$
Therefore $\begin{pmatrix}
                   Y \\
                 -Z \\
                \end{pmatrix}$
is the eigenvector corresponding to the eigenvalue $-\lambda$ of $A$. Assume that $\lambda_1(A),\ldots,\lambda_r(A)$ are the positive eigenvalues of $A$, we have that
$$
|\lambda I_n-A|=\lambda^{n-2r}\prod_{i=1}^{r}(\lambda^2-\lambda_i^2(A)).
$$
By the equation
$$
\begin{pmatrix}
  \lambda I_{n_1} & -X \\
  -X^T & \lambda I_{n_2} \\
\end{pmatrix}\begin{pmatrix}
  \lambda I_{n_1} & 0 \\
  -X^T & \lambda I_{n_2} \\
\end{pmatrix}=\begin{pmatrix}
  \lambda^2 I_{n_1}-XX^T & -X \\
  0 & \lambda I_{n_2} \\
\end{pmatrix}.
$$
We have that
$$
\lambda^n \lambda^{n-2r}\prod_{i=1}^{r}(\lambda^2-\lambda_i^2(A))=\lambda^{n_2}|\lambda^2-XX^T|=
\lambda^{n_2}\prod_{i=1}^{n_1}(\lambda^2-\lambda_i(XX^T)).
$$
Therefore we have that
$$
\mathcal{E}(A)=2\sum_{i=1}^r\lambda_i(A)=2\sum_{i=1}^{n_1}\sqrt{\lambda_i(XX^T)},
$$
as desired.
\end{proof}
By Eq. \ref{eq3.1} and Lemma \ref{lemma3.3}, we have that
\begin{eqnarray*}
  \mathcal{E}(\overline{S}_{n,2}(p,q)) &=& 2\sum_{i=1}^{n_1}\sqrt{\lambda_i(XX^T)} \\
   &=&  2\sqrt{n_1}\sum_{i=1}^{n_1}\sqrt{\lambda_i(\frac{1}{n_1}XX^T)}\\
   &=& 2n_1^{3/2}\int\sqrt{x}d\Phi_{\frac{1}{n_1}XX^T}
\end{eqnarray*}
With a similar analysis with Section 2, by Lemma \ref{lemma3.2}, we have that
\begin{eqnarray*}
  \lim_{n\rightarrow\infty}\int\sqrt{x}d\Phi_{\frac{1}{n_1}XX^T} &=& \int\sqrt{x}dF_y(x)
  \ \ a.s. \\
   &=&  \int_a^b\frac{\sqrt{x}}{2\pi xy\sigma^2}\sqrt{(b-x)(x-a)}dx\\
   &=& \frac{1}{\pi y\sigma^2}\int_{\sqrt{a}}^{\sqrt{b}}\sqrt{(b-x^2)(x^2-a)}dx\\
   &=& \frac{1}{\pi y\sigma^2}\int_{|1-\sqrt{y}|}^{1+\sqrt{y}}\sqrt{4y-(1+y-x^2)^2}dx
\end{eqnarray*}
Set $\eta=\int_{|1-\sqrt{y}|}^{1+\sqrt{y}}\sqrt{4y-(1+y-x^2)^2}dx$, we have that
\begin{eqnarray*}
\mathcal{E}(\overline{S}_{n,2}(p,q))&=&\left(\frac{2\eta}{\pi y (p+q-(p-q)^2)}+o(1)\right)(n\beta_1)^{3/2} \ \ \ a.s. \\
&=& \left(\frac{2\beta_2\sqrt{\beta_1}\eta}{\pi (p+q-(p-q)^2)}+o(1)\right)n^{3/2}.
\end{eqnarray*}
Employing Eq. \ref{eq3.1} and Lemma \ref{lemma5}, we have that
$$
\mathcal{E}(\overline{S}_{n,2}(p,q))-\mathcal{E}((p-q)A(K_{n_1,n_2}))\le
\mathcal{E}(S_{n,2}(p,q))\le
\mathcal{E}(\overline{S}_{n,2}(p,q))-\mathcal{E}((p-q)A(K_{n_1,n_2})).
$$
Together with the fact that $$\mathcal{E}((p-q)A(K_{n_1,n_2}))=2|p-q|\sqrt{n_1n_2}
=2n\sqrt{\beta_1\beta_2}|p-q|,$$
we get
$$
\mathcal{E}(S_{n,2}(p,q))=\left(\frac{2\beta_2\sqrt{\beta_1}\eta}{\pi (p+q-(p-q)^2)}+o(1)\right)n^{3/2} \ \ \ a.s.
$$
Therefore, the following theorem is relevant.
\begin{theorem}\label{th3}
Almost every random bipartite signed graph $G_{n;\beta_1,\beta_2}(p,q)$ with $\beta_1/\beta_2\rightarrow y\in (0, \infty)$ enjoys
$$\mathcal{E}(G_{n;\beta_1,\beta_2}(p,q))=\left(\frac{2\beta_2\sqrt{\beta_1}\eta}{\pi (p+q-(p-q)^2)}+o(1)\right)n^{3/2},$$
where $\eta=\int_{|1-\sqrt{y}|}^{1+\sqrt{y}}\sqrt{4y-(1+y-x^2)^2}dx.$
\end{theorem}
\textbf{Remark 2.}If $q=0$, it is obvious that $\mathcal{G}_{n;\beta_1,\beta_2}(p,0)=\mathcal{G}_{n;\beta_1,\beta_2}(p)$. Thus by Theorem \ref{th3}, we can get the energy of random bipartite graph $G_{n;\beta_1,\beta_2}(p)$, which can be also get in \cite{Li1}.
\begin{coro}
Almost every random bipartite graph $G_{n;\beta_1,\beta_2}(p)$ enjoys the equation as follows:
$$
\mathcal{E}(A(G_{n;\beta_1,\beta_2}(p)))=
\left(\frac{2\beta_2\sqrt{\beta_1}\eta}{\pi (p-p^2)}+o(1)\right)n^{3/2}. $$
\end{coro}
For signed random bipartite graph $G_{n;\beta_1,\beta_2}(p,q)\in \mathcal{G}_{n;\beta_1,\beta_2}(p,q)$, the underlying unsigned graph  belongs to $\mathcal{G}_{n;\beta_1,\beta_2}(p+q)$. It is obvious that that with probability 1
\begin{eqnarray*}
  &&\frac{\mathcal{E}(A(G_{n;\beta_1,\beta_2}(p,q)))
  -\mathcal{E}(A(G_{n;\beta_1,\beta_2}(p+q)))}{n^{3/2}}\\ &=&
  \left(\frac{2\beta_2\sqrt{\beta_1}\eta}{\pi ((p+q)-(p-q)^2)}-\frac{2\beta_2\sqrt{\beta_1}\eta}{\pi ((p+q)-(p+q)^2)}+o(1)\right)\\
  &=&\left(\frac{-8pq\beta_2\sqrt{\beta_1}\eta}{\pi ((p+q)-(p-q)^2)(p+q)-(p+q)^2)}+o(1)\right).
\end{eqnarray*}
Thus we have that if $p\neq0$ and $q\neq0$, the energy of random signed bipartite graph $G_{n;\beta_1,\beta_2}(p,q))$ is less than the energy of its underlying unsigned graph and furthermore, we can get the following result.
\begin{theorem}
Almost every random bipartite signed graph $G_{n;\beta_1,\beta_2}(p,q)$ is unbalanced.
\end{theorem}
\section{Conclusion}
In this paper, we obtain an exact estimate of energy for the random signed graph $G_n(p,q)$ and we establish lower and upper bounds to the energy of random multipartite signed graphs $G_{n;\beta_1,\ldots,\beta_k}$. Furthermore, by comparing the energy of random signed graph (resp. random bipartite signed graph) with the energy of its underlying graph, we deduce that a.s. $G_n(p,q)$ (resp. $G_{n;\beta_1,\beta_2}(p,q)$) are unbalanced.

\end{document}